\documentclass[12pt]{amsart}

\usepackage{amssymb}
\usepackage{bm}
\usepackage{graphicx}
\usepackage{psfrag}
\usepackage{hyperref}
\hypersetup{colorlinks=true, linkcolor=blue, citecolor=red, urlcolor=wine}
\usepackage{color}
\usepackage{url}
\usepackage{algpseudocode}
\usepackage{fancyhdr}
\usepackage{xy}
\input xy
\xyoption{all}
\usepackage{stmaryrd}
\usepackage{calrsfs}
\usepackage[dvipsnames]{xcolor}

\newtheorem*{maintheorem*}{Main Theorem}
\newtheorem{theorem}{Theorem}[section]
\newtheorem{prop}[theorem]{Proposition}

\newtheorem{lemma}[theorem]{Lemma}
\newtheorem{cor}[theorem]{Corollary}
\theoremstyle{definition}
\newtheorem{definition}[theorem]{Definition}

\newtheorem{example}[theorem]{Example}
\numberwithin{equation}{section}

\newcommand{\nn}{\mathbb{N}}
\newcommand{\pp}{\mathbb{P}}
\newcommand{\qq}{\mathbb{Q}}
\newcommand{\rr}{\mathbb{R}}
\newcommand{\uu}{\mathcal{U}}
\newcommand{\zz}{\mathbb{Z}}

\newcommand{\gp}{\text{gp}}

\providecommand\ldb{\llbracket}
\providecommand\rdb{\rrbracket}

\keywords{abstract divisibility, maximal common divisor, MCD monoid, $k$-MCD monoid, Puiseux monoid}


\begin{document}
	\mbox{}
	\title{On maximal common divisors in Puiseux monoids}
	
	\author{Evin Liang}
	\address{Lexington School\\Lexington, MA 02421}
	\email{ez1liang@gmail.com}
	
	\author{Alexander Wang}
	\address{Millburn School\\Millburn, NJ 07041}
	\email{abwang07@gmail.com}
	
	\author{Lerchen Zhong}
	\address{Westwood School\\Austin, TX 78750}
	\email{lerchenzhong@gmail.com}
	
	
	\begin{abstract}
		Let $M$ be a commutative monoid. An element $d \in M$ is called a maximal common divisor of a nonempty subset $S$ of $M$ if $d$ is a common divisor of $S$ in $M$ and the only common divisors in $M$ of the set $\big\{ \frac{s}d : s \in S \big\}$ are the units of $M$. In this paper, we investigate the existence of maximal common divisors in rank-$1$ torsion-free commutative monoids, also known as Puiseux monoids. We also establish some connections between the existence of maximal common divisors and both atomicity and the ascending chain condition on principal ideals for the monoids we investigate here.
	\end{abstract}

\maketitle

\section{Introduction}
\label{sec:intro}

Let $M$ be a commutative monoid (i.e., a commutative semigroup with an identity element). Maximal common divisors can be defined in the abstract setting of commutative monoids as follows: a common divisor $d \in M$ of a nonempty subset $S$ of $M$ is called a maximal common divisor of $S$ provided that the only common divisors of the set $\big\{ \frac{s}d : s \in S \big\}$ are the units (i.e., invertible elements) of~$M$. In this paper, we study the existence of maximal common divisors in the setting of commutative monoids, focusing on the class consisting of all rank-$1$ torsion-free monoids (that are not nontrivial groups), also known as Puiseux monoids. In this class of monoids, we investigate the existence of maximal common divisors in connection to both atomicity and the ascending chain condition on principal ideals.
\smallskip

Following Cohn~\cite{pC68}, we say that the monoid $M$ is atomic if every nonunit of $M$ factors into finitely many atoms (i.e., irreducible elements). In addition, the monoid~$M$ is said to satisfy the ascending chain condition on principal ideals (ACCP) provided that every ascending chain of principal ideals eventually becomes constant. An integral domain (i.e., a commutative ring with identity and no nonzero zero-divisors) is atomic (resp., satisfies the ACCP) if its multiplicative monoid of nonzero elements is atomic (resp., satisfies the ACCP). It is well known that every monoid (or integral domain) that satisfies the ACCP is atomic. The distinction from atomicity and the ACCP is subtle and challenging to distinguish in the context of integral domains. The first atomic integral domain not satisfying the ACCP was constructed by Grams~\cite{aG74}, correcting the assertion made by Cohn in~\cite{pC68} that every atomic integral domain satisfies the ACCP. Since then, various alternative constructions of atomic integral domains not satisfying the ACCP have been provided in the literature, each of them revealing a novel nontrivial approach: for instance, see the constructions in~\cite{mR93,aZ82} and the more recent constructions in~\cite{BC19,BGLZ24,GL23,GL23a,GL22}. Atomic non-commutative rings not satisfying the ACCP have been recently constructed in~\cite{BBNS23}.
\smallskip

The construction of an atomic integral domain not satisfying the ACCP given by Roitman in~\cite{mR93} provided a negative answer to the ascent of atomicity from integral domains to their corresponding polynomial extensions, which was an open question posed in~\cite{AAZ90} by Anderson, Anderson, and Zafrullah. Roitman's construction is based on the notion of an MCD domain. A commutative monoid (resp., an integral domain) is called an MCD monoid (resp., an MCD domain) if every nonempty finite subset of the same has a maximal common divisor. Under the existence of maximal common divisors, atomicity does ascend to polynomial extensions (see \cite[Proposition~1.1]{mR93}). In the same paper, Roitman introduces the notion of $k$-MCD (for every positive integer~$k$), which means that every subset of cardinality $k$ has a maximal common divisor. In this paper we consider the strong MCD property, which is a natural (stronger) version of MCD property: we say that a commutative monoid is strongly MCD if every nonempty subset has a maximal common divisor. Both the MCD and strong MCD properties play a relevant role in this paper.
\smallskip

In Section~\ref{sec:background} we introduce most of the notation and terminology we shall be using later. We also provide a quick revision of the background needed to fully understand the results we establish here.
\smallskip

In Section~\ref{sec:two monoids} we explore the connection between the ACCP and the strong MCD property. It is known, although we were unable to find a reference in the literature, that every commutative monoid that satisfies the ACCP is an MCD monoid. We start Section~\ref{sec:two monoids} by arguing the stronger result that every commutative monoid that satisfies the ACCP is strongly MCD. Then we prove that in the class of countable monoids, which includes all Puiseux monoids, the ACCP and the strong MCD property are equivalent. We finish the same section discussing two examples of atomic Puiseux monoids that have appeared in recent literature of factorization theory: the so-called Grams' monoid and the Puiseux monoid generated by the terms of the sequence $\big( \frac{1}{p_n p_{n+2}} \big)_{n \ge 1}$, where $p_n$ is the $n$-th prime. As we will argue later, these two monoids (which motivate the results we establish in Sections~\ref{sec:generalization of Grams' monoid} and~\ref{sec:2-prime reciprocal monoids}) are examples of countable monoids that are MCD but not strongly MCD.
\smallskip

The essential component in Grams' classical and first-known construction of an atomic integral domain not satisfying the ACCP is Grams' monoid, which is an atomic Puiseux monoid not satisfying the ACCP as it contains the valuation monoid $\nn_0[\frac12]$ as a submonoid. Motivated by Grams' monoid, in Section~\ref{sec:generalization of Grams' monoid} we introduce and study some monoid extensions $N \subseteq M$, where both $N$ and $M$ are Puiseux monoids and $M$ generalizes Grams' monoid. We first prove that, as in the case of Grams' monoid, the monoid $M$ is atomic, and if $N$ is a valuation monoid then $M$ does not satisfy the ACCP. We also prove that $M$ is an MCD monoid provided that $N$ is a valuation monoid and, as $M$ is countable, it follows that when $N$ is a valuation monoid, the monoid $M$ is an MCD that is not strongly MCD. Following Eftekhari and Khorsandi~\cite{EK18}, we say that a commutative monoid is MCD-finite if every nonempty finite subset has only finitely many maximal common divisors, up to associates. We conclude Section~\ref{sec:generalization of Grams' monoid} proving that, for the studied monoid extensions $N \subseteq M$, the monoid $M$ is MCD-finite if and only if its submonoid $N$ is MCD-finite.
\smallskip

As we do with Grams' monoid in Section~\ref{sec:generalization of Grams' monoid}, in Section~\ref{sec:2-prime reciprocal monoids} we study a class of atomic Puiseux monoids that we call $2$-prime reciprocal. Monoids in this class generalize the monoid $M_\pp$ generated by the sequence $\big( \frac{1}{p_n p_{n+2}} \big)_{n \ge 1}$. A larger class of Puiseux monoids (not all of them atomic) was previously investigated in~\cite{AGH21} under the term `weak reciprocal Puiseux monoids' (this explains why we chose the term `$2$-prime reciprocal' for the monoids investigated in Section~\ref{sec:2-prime reciprocal monoids}). After quickly showing that $2$-prime reciprocal Puiseux monoids are atomic but do not satisfy the ACCP, we prove the main result of Section~\ref{sec:2-prime reciprocal monoids}: every $2$-prime reciprocal Puiseux monoid is an MCD monoid. Gotti and Li already proved in~\cite{GL23a} that $M_\pp$ is strongly atomic (that is, atomic and $2$-MCD). Thus, as $2$-prime reciprocal Puiseux monoids are countable, we obtain another class of atomic MCD monoids that are not strongly MCD.
\smallskip

In Section~\ref{sec:atomic not strongly atomic} we construct an atomic monoid that is not $2$-MCD, and so it is not an MCD monoid as the rest of the atomic monoids we investigate in previous sections. The search for atomic monoids that are not MCD was initiated by Roitman in~\cite{mR93}, where he constructed, for each $k \in \nn$, an atomic integral domain that is $k$-MCD but not $(k+1)$-MCD. Restricted to the class of finite-rank monoids, which is the one concerning this paper, the first atomic monoid that is not $2$-MCD (and so is not an MCD) first appeared in~\cite[Example~4.3]{GV23}; it is a monoid with rank-$3$. In the same paper, the authors left the open question of whether there exists an atomic Puiseux monoid that is not $2$-MCD, and the same question was positively answered by Gonzalez et al.~\cite{GLRRT24}. More recently, Gotti and Rabinovitz generalized this by constructing, for each $k \in \nn$, an atomic Puiseux monoid that is $k$-MCD but not $(k+1)$-MCD, and they did so using a refined version of the atomization technique introduced in~\cite{GL23} (see \cite{GR23}). Although the atomic non-$2$-MCD monoid we construct in Section~\ref{sec:atomic not strongly atomic} has rank~$2$ (and so is not a Puiseux monoid), it is worth emphasizing that our construction not only seems to be simpler but is based on a new idea that may have further applications.

\bigskip
\section{Background}
\label{sec:background}

\smallskip
\subsection{General Notation}

Let us first introduce some general notation we shall be using throughout this paper. We let $\nn$ and $\nn_0$ denote the set of positive integers and the set of nonnegative integers, respectively. In addition, we let $\mathbb{P}$ denote the set of primes. As it is customary, we let $\mathbb{Q}$ and $\rr$ stand for the set of rational numbers and the set of real numbers, respectively. For a subset $S$ of $\rr$ and $r \in \rr$, we set $S_{\ge r} := \{s \in S : s \ge r\}$ and $S_{> r} := \{s \in S : s > r\}$. For $b,c  \in \mathbb Z$, we set
\[
	\llbracket b,c \rrbracket := \{n \in \zz : b \le n \le c\},
\]
allowing the discrete interval $\ldb b,c \rdb$ to be empty when $b > c$. For a nonzero $q \in \qq$, we let $\mathsf{n}(q)$ and $\mathsf{d}(q)$ denote, respectively, the unique $n \in \zz$ and $d \in \nn$ such that $q = n/d$ and $\gcd(n,d) = 1$. For $p \in \pp$ and a nonzero $n \in \zz$, we let $v_p(n)$ denote the maximum $m \in \nn_0$ such that $p^m \mid n$. Then for each $p \in \pp$, we let $v_p \colon \qq \setminus \{0\} \to \zz$ denote the $p$-adic valuation map, which means that for every nonzero $q \in \qq$
\[
	v_p(q) = v_p(\mathsf{n}(q)) - v_p(\mathsf{d}(q)).
\]

\medskip
\subsection{Commutative Monoids}

Let us introduce now some terminology about commutative monoids. Throughout this paper we tacitly assume that every monoid is commutative and cancellative. Let~$M$ be a monoid that is additively written (and with identity element denoted by~$0$). We say that $M$ is \emph{nontrivial} if $M \neq \{0\}$. The group consisting of all the invertible elements of $M$ (also called units) will be denoted by $\uu(M)$, and we say that $M$ is \emph{reduced} if $\uu(M)$ is the trivial group. A subset~$N$ of~$M$ is called a \emph{submonoid} if $N$ contains~$0$ and is closed under the operation of~$M$. The arbitrary intersection of submonoids of $M$ is also a submonoid of $M$. If $S$ is a subset of $M$, then we let $\langle S \rangle$ denote the submonoid of~$M$ generated by $S$; that is, $\langle S \rangle$ is the intersection of all the submonoids of~$M$ containing~$S$. The \emph{Grothendieck group} of $M$, denoted by $\gp(M)$, is the abelian group consisting of all the formal differences of elements of $M$. As $M$ is assumed to be cancellative, we can identify $M$ with a submonoid of $\gp(M)$ (indeed, it turns out that $\gp(M)$ is the unique abelian group up to isomorphism that minimally contains an isomorphic copy of~$M$). The monoid~$M$ is called \emph{torsion-free} provided that $\gp(M)$ is a torsion-free abelian group or, equivalently, for each $n \in \nn$ and $b,c \in M$, the equality $nb = nc$ implies that $b=c$. 
\smallskip

A non-invertible element $a \in M$ is called an \emph{atom} (or \emph{irreducible}) if for all $b,c \in M$ the equality $a = b+c$ implies that either $b \in \uu(M)$ or $b \in \uu(M)$. The set consisting of all the atoms of~$M$ is denoted by $\mathcal{A}(M)$. Following Cohn~\cite{pC68}, we say that the monoid $M$ is \emph{atomic} if every element in $M \setminus \uu(M)$ can be written as a sum of finitely many atoms (allowing repetitions). The monoid $M$ is called a \emph{unique factorization monoid} (UFM) if it is atomic and every non-invertible element of $M$ can be written as a sum of finitely many atoms in a unique way (up to order and associates). For progress on atomicity in the setting of integral domains, see the recent survey~\cite{CG24} by Coykendall and Gotti. Ascending chains of principal ideals are often studied in connection to atomicity. A subset $I$ of $M$ is called an \emph{ideal} if
\[
	I+M := \{b+c : b \in I \text{ and } c \in M\} \subseteq I.
\]
An ideal of $M$ is called \emph{principal} if it has the form $b + M$ for some $b \in M$. The monoid $M$ is called a \emph{valuation monoid} if for all $b,c \in M$ either $b + M \subseteq c + M$ or $c + M \subseteq b + M$. A sequence $(I_n)_{n \ge 1}$ of subsets of $M$ is called \emph{ascending} if $I_n \subseteq I_{n+1}$ for all $n \in \nn$, while a sequence $(J_n)_{n \ge 1}$ of subsets of $M$ is said to \emph{stabilize} if we can take $N \in \nn$ such that $J_n = J_N$ for all $n \in \nn $ with $n \ge N$. Then we say that $M$ satisfies the \emph{ascending chain condition on principal ideals} (ACCP) if every ascending chain of principal ideals eventually stabilizes. Every monoid satisfying the ACCP is atomic \cite[Proposition~1.1.4]{GH06b}, but the converse is not true in general. Some examples of atomic monoids that do not satisfy the ACCP, along with recent progress on the connection between the ACCP and atomicity, are provided in~\cite{GL23}. Also, we will discuss further examples throughout this paper. Finally, nontrivial elementary examples of monoids satisfying the ACCP are the additive submonoids of $\rr_{\ge 0}$ whose corresponding sets of nonzero elements are bounded below by a positive real number \cite[Proposition~4.5]{fG19}.
\smallskip

Additive submonoids of $\qq_{\ge 0}$ are the most relevant algebraic objects in this paper; they are called \emph{Puiseux monoids}. The atomicity and arithmetic of factorizations of Puiseux monoids have been systematically investigated during the past decade (\cite{CGG20a,GR23} provides a friendly survey in this direction). The \emph{rank} of the monoid $M$ is, by definition, the rank of its Grothendieck group or, equivalently, the dimension of the vector space $\qq \otimes_\zz \gp(M)$ over the field $\qq$ obtained by tensoring the $\zz$-modules $\qq$ and $\gp(M)$. Observe that every nontrivial Puiseux monoid has rank~$1$. It is known that nontrivial Puiseux monoids account, up to isomorphism, for all rank-$1$ torsion-free monoids that are not abelian groups \cite[Theorem~3.12]{GGT21} (see also \cite[Section~24]{lF70} and \cite[Theorem~2.9]{rG84}). Finally, observe that every finitely generated Puiseux monoid is an additive submonoid of $\nn_0$ (or a numerical monoid) up to isomorphism, and so it must satisfy the ACCP. More generally, every additive submonoid of $\rr_{\ge 0}$ (and so every Puiseux monoid)
\smallskip

\medskip
\subsection{Divisibility}

Even though it is more natural to write monoids multiplicatively when dealing with notions related to divisibility (as we did in the introduction), we will keep writing the monoid $M$ additively. The reason for this is that the monoids we are primarily interested in the scope of this paper are Puiseux monoids, which are additive by nature.
\smallskip

For $b,c \in M$, we say that $c$ \emph{divides} $b$ and write $c \mid_M b$ provided that $b = c+d$ for some $d \in M$. For a nonempty subset $S$ of $M$, an element $d$ is called a \emph{common divisor} of $S$ if $d \mid_M s$ for all $s \in S$, in which case, we let $S-d$ denote the set $\{s-d : s \in S\}$. A \emph{greatest common divisor} (GCD) of a nonempty subset $S$ of $M$ is a common divisor $d \in M$ of $S$ such that any other common divisor of $S$ divides~$d$ in $M$. A \emph{maximal common divisor} (MCD) of a nonempty subset $S$ of $M$ is a common divisor $d \in M$ of $S$ such that the only common divisors of $S-d$ are the invertible elements of~$M$. In addition, the monoid $M$ is called a \emph{GCD monoid} (resp., an \emph{MCD monoid}) provided that every nonempty finite subset of~$M$ has a GCD (resp., an MCD). It is clear that every greatest common divisor is a maximal common divisor, and so every GCD monoid is an MCD monoid. For $k \in \nn$, we say that $M$ is a $k$-\emph{MCD monoid} if every subset of $M$ of cardinality $k$ has a maximal common divisor. Clearly, every monoid is a $1$-MCD monoid. Also, observe that $M$ is an MCD monoid if and only if $M$ is a $k$-MCD monoid for every $k \in \nn$. The notion of $k$-MCD was introduced by Roitman in~\cite{mR93}. Finally, we say that~$M$ is an \emph{MCD-finite} monoid if every nonempty finite subset of $M$ has only finitely many maximal common divisors up to associates. The notions of an MCD, $k$-MCD (for every $k \in \nn$), and MCD-finite are fundamental in the scope of this paper.
\smallskip

We conclude this section by introducing a property in the class of commutative (but not necessarily cancellative) monoids, which is a natural strengthening of the MCD property.
\begin{definition}
	Let $M$ be a commutative monoid. We say that $M$ is \emph{strongly MCD} or a \emph{strong MCD monoid} if every nonempty subset of $M$ has a maximal common divisor.
\end{definition}

It follows directly from the definition that every strong MCD monoid is an MCD monoid.
As we are about to prove in Section~\ref{sec:two monoids}, every monoid that satisfies the ACCP is strongly MCD, and both properties are equivalent for countable monoids (and so for Puiseux monoids).

\bigskip
\section{The Strong MCD Property and the ACCP}
\label{sec:two monoids}

In this section, we consider the strong MCD property. As the following proposition indicates, every monoid satisfying the ACCP is a strong MCD monoid.

\begin{prop} \label{prop:ACCP implies strong MCD}
	Let $M$ be a commutative monoid. If $M$ satisfies the ACCP, then $M$ is a strong MCD monoid.
\end{prop}

\begin{proof}
	Suppose that the monoid $M$ satisfies the ACCP. Now assume for the sake of contradiction that $M$ is not strongly MCD, and take a nonempty $S \subseteq M$ such that $S$ has no maximal common divisor. Then there is a common divisor $d_1$ of $S$ that is not invertible. Now observe that if the set $S_1 := S - d_1$ had a maximal common divisor $d$, then $d_1 + d$ would be a maximal common divisor of $S$, which is not possible. Thus, the set $S_1$ does not have a maximal common divisor in $M$. Thus there exists a non-invertible $d_2 \in M$ that is a common divisor of $S_1$, and in the same way we obtain that $S_2 := S_1 - d_2 = S - (d_1 + d_2)$ is a subset of $M$ that does not have any maximal common divisor. Continuing in this fashion, we can obtain a sequence of non-invertible elements $(d_n)_{n \ge 1}$ such that $d_1 + \dots + d_n$ is a common divisor of $S$ for every $n \in \nn$, and so after fixing $s_0 \in S$ and setting $s_n := d_1 + \dots + d_n$ for every $n \in \nn$, we obtain that $(s_0 - s_n + M)_{n \ge 1}$ is an ascending chain of principal ideals that does not stabilize: indeed, $(s - s_n) - (s - s_{n+1}) = d_{n+1}$ is a non-invertible element of $M$ for every $n \in \nn$.
\end{proof}

The converse of Proposition~\ref{prop:ACCP implies strong MCD} does not hold in general. The next example sheds some light upon this observation.

\begin{example}
	The additive monoid $\rr_{\ge 0}$ is strongly MCD but does not satisfy the ACCP. The sequence $\left(\frac1{2^n}+\rr_{\ge 0}\right)_{n \ge 1}$ is an ascending chain of principal ideals of $\rr_{\ge 0}$ that does not stabilize, so $\rr_{\ge 0}$ does not satisfy the ACCP. For any subset $S\subseteq\rr_{\ge 0}$, the completeness property of the reals implies there exists a nonnegative real $x=\inf S$ such that $x$ is the largest real that is less than or equal to every element of $S$. Then, $x$ is an MCD of $S$ because if $x+d$ divides every element of $S$, then $\inf S\geq x+d$, which is a contradiction. This implies $\rr_{\ge 0}$ is strongly MCD.
\end{example}

However, the converse of Proposition~\ref{prop:ACCP implies strong MCD} does hold for countable commutative monoids.

\begin{prop} \label{prop:an ACCP countable monoid is strong MCD}
	Let $M$ be a countable commutative monoid. If $M$ is strongly MCD, then $M$ must satisfy the ACCP. 
\end{prop}

\begin{proof}
	Let $M$ be a monoid that does not satisfy the ACCP. We aim to show that $M$ is not strongly MCD. Let $(a_n + M)_{n \ge 1}$ be an ascending chain of principal ideals  of $M$ that does not stabilize. Define $b_n := a_n - a_{n+1}$ for all positive integers $n$. As the chain $(a_n + M)_{n \ge 1}$ does not stabilize, we can assume, without loss of generality, that $b_n$ is never invertible. Let $(c_n)_{n \ge 1}$ be a sequence enumerating the elements of $M$. Suppose for simplicity that $c_1 = 0$.
	
	We define a set $S$ of positive integers with the property that the (infinite) set
	\[
		T: = \left\{ \sum_{i \in S, \, i<k} b_i + a_k : k \in \mathbb N \right\}
	\]
	has no maximal common divisor in $M$. Let $e_k=\sum_{i \in S, \, i<k} b_i + a_k$ for all $k\in\mathbb N$. Note that $e_{k}-e_{k+1}$ is either $ (a_k-a_{k+1})- b_k = 0$ or $a_k-a_{k+1}=b_k$, and hence $e_{k+1}$ divides $e_k$ for all $k \in \mathbb{N}$. We construct $S$ in stages by defining sets $S_i$ and an increasing sequence of positive integers $\ell_i$ for each positive integer $i$ with the property that $S_i$ and $S_{i+1}$ contain the same integers at most $\ell_i$, and every element of $S_i$ is at most $\ell_i$. Then $S=\bigcup_{i\in\mathbb N}S_i$. Define $S_i$ and $\ell_i$ inductively as follows.
	\begin{itemize}
		\item $S_1=\{1\}$ and $\ell_1=1$. Then $c_1$ cannot be a maximal common divisor of $T$ because $c_1+b_1$ must be a common divisor of $T$.
		\smallskip
		
		\item If $c_{i+1}$ divides every element of the form $\sum_{j\in S_i}b_j+a_k$ for positive integers $k$, then we set $S_{i+1}=S_i\cup\{\ell_i+1\}$ and $\ell_{i+1}=\ell_i+1$. Then $c_{i+1}+b_{\ell_i+1}$ divides $\sum_{j\in S_{i+1}}b_j+a_k$ for all positive integers $k$. Since
        \[
            \sum_{j\in S_{i+1}}b_j+a_k=\sum_{j\in S,i\le\ell_{i+1}}b_j+a_k
        \]
        divides $e_k$ for any $k>\ell_{i+1}$, and $e_y$ divides $e_x$ for any $y>x$, it follows that $c_{i+1}+b_{\ell_i+1}$ is a common divisor of $T$. So $c_{i+1}$ cannot be a maximal common divisor of $T$.
		\smallskip
		
		\item If $c_{i+1}$ does not divide every element of the form $\sum_{j\in S_i}b_j+a_k$ for positive integers $k$, then we let $m$ denote the least positive integer greater than $\ell_i$ for which $c_{i+1}$ does not divide $\sum_{j\in S_i}b_j+a_{m+1}$. Then set $S_{i+1}=S_i$ and $\ell_{i+1}=m$. Then
  \[\sum_{j\in S_i}b_j+a_{m+1}=\sum_{j\in S_{i+1}}b_j+a_{\ell_{i+1}+1}=\sum_{j\in S,j<\ell_{i+1}+1}b_j+a_{\ell_{i+1}+1}=e_{\ell_{i+1}+1},\] 
  where $c_{i+1}$ does not divide the leftmost expression and thus does not divide the rightmost expression either. But $e_{\ell_{i+1}+1}$ is an element of $T$, so $c_{i+1}$ cannot possibly be a common divisor of $T$.
	\end{itemize}
	
	Thus, the resulting set $T$ has no maximal common divisor in $M$, and so $M$ is not strongly MCD.
\end{proof}

Combining Propositions~\ref{prop:ACCP implies strong MCD} and~\ref{prop:an ACCP countable monoid is strong MCD}, we obtain the following corollary.

\begin{cor} \label{cor:ACCP equivalent to strong MCD for PMs}
	A Puiseux monoid satisfies the ACCP if and only if it is strongly MCD. 
\end{cor}



We will conclude this section briefly discussing the two examples of Puiseux monoids that motivated Sections~\ref{thm:when Grams-like monoids are MCD} and~\ref{sec:2-prime reciprocal monoids}. They are not only the prototypical examples for the two classes of monoids we investigate in the next two sections, but they also provide simple negative answers for the following two natural questions.

\begin{enumerate}
    \item It follows directly from the corresponding definitions that every strong MCD monoid is an MCD monoid. It is natural to wonder whether the converse holds. The answer is negative: the Puiseux monoids discussed in Examples~\ref{ex:Grams monoid} and~\ref{ex:prototypical 2-prime reciprocal} are both MCD monoids that are not strongly MCD.
    \smallskip

    \item It is well known that every atomic GCD monoid is a UFM. It is natural to wonder if one can improve the same result by replacing the GCD condition by the weaker condition of being an MCD monoid. The answer is negative: the Puiseux monoids discussed in Examples~\ref{ex:Grams monoid} and~\ref{ex:prototypical 2-prime reciprocal} are both atomic MCD monoids that do not satisfy the ACCP and, therefore, they are not UFMs.
\end{enumerate}

We start with the monoid in Grams' construction of the first atomic integral domain not satisfying the ACCP.

\begin{example} \label{ex:Grams monoid}
	For each $n \in \nn$, let $p_n$ be the $n$-th odd prime, and consider the Puiseux monoid defined as follows:
	\[
		M := \Big\langle \frac1{2^{n-1} p_n} : n \in \nn \Big\rangle.
	\]
	It is well known and not difficult to verify that $\mathcal{A}(M) = \big\{ \frac1{2^{n-1} p_n} : n \in \nn \big\}$, and so that $M$ is atomic. Also, observe that $N := \zz[\frac12]_{\ge 0} = \big\langle \frac1{2^n} : n \in \nn \big\rangle$ is a submonoid of $M$ that is a valuation monoid. Therefore $\big(\frac1{2^n} + M \big)_{n \ge 1}$ is an ascending chain of principal ideals of $M$ that does not stabilize. As a result, $M$ is a rank-$1$ torsion-free atomic monoid that does not satisfy the ACCP. Therefore it follows from Corollary~\ref{cor:ACCP equivalent to strong MCD for PMs} that $M$ is not strongly MCD. Although $M$ is not strongly MCD (and does not satisfy the ACCP), it is indeed an MCD monoid: this will follow as a special case of a more general scenario we will consider in Proposition~\ref{prop:<1/(x_np_n)> is not ACCP}. The atomic structure and arithmetic of factorizations of classes of atomic Puiseux monoids generalizing $M$ have been investigated in the recent papers~\cite[Section~3]{GL23} and \cite[Section~4]{CJMM24}.
 
    Let us briefly explain the role played by the Puiseux monoid $M$ in Grams' construction of the first atomic integral domain not satisfying the ACCP. Let $F$ be a field, and consider the monoid algebra $F[M]$ consisting of all polynomial expressions with coefficients in $F$ and exponents in the monoid $M$ (with addition and multiplication defined as for standard polynomials). Now observe that
    \[
        S := \{f \in F[M] : f(0) \neq 0 \}
    \]
    is a multiplicative subset of the integral domain $F[M]$. Let $R$ be the localization of $F[M]$ at $S$. It was first proved by Grams~\cite{aG74} that $R$ is an atomic integral domain that does not satisfy the ACCP, disproving a wrong assertion made in~\cite{pC68} that in the setting of integral domains, being atomic and satisfying the ACCP were equivalent conditions.
    
    Finally, it is still open whether there exists a field $F$ such that the monoid algebra $F[M]$ of $M$ over $F$ is not atomic. In general, the property of being atomic does not ascend from finite-rank torsion-free monoids to their corresponding monoid algebras over fields (this was proved in~\cite{CG19} for any field of prime cardinality and, more recently, in~\cite{GR23} for any integral domain). In the same direction, it is also known that the property of being atomic does not ascend from integral domains to their corresponding polynomial rings, and this was proved by Roitman~\cite{mR93} back in 1993.
\end{example}

%
%
%

The second Puiseux monoid we proceed to discuss is the prototype motivating Section~\ref{sec:2-prime reciprocal monoids}.


\begin{example} \label{ex:prototypical 2-prime reciprocal}
    For each $n \in \nn$, we now let $p_n$ denote the $n$-th prime. Consider the Puiseux monoid defined as follows:
    \[
        M := \bigg\langle \frac{1}{p_n p_{n+2}} : n \in \nn \bigg\rangle.
    \]
    It is not difficult to argue that $M$ is an atomic monoid whose atoms are precisely the defining generators: $\mathcal{A}(M) = \big\{ \frac1{p_n p_{n+2}} : n \in \nn \big\}$. Also, as shown in \cite[Example 3.9]{GL23a}, the monoid $M$ does not satisfy the ACCP as the fact that
    \[
        \frac1{p_n} = \frac1{p_{n+2}} + (p_{n+2} - p_n) \frac1{p_n p_{n+2}}
    \]
    for every $n \in \nn$ implies that both
    \[
        \bigg( \frac1{p_{2n-1}} + M \bigg)_{n \ge 1} \quad \text{ and } \quad \bigg( \frac1{p_{2n}} + M \bigg)_{n \ge 1}
    \]
    are ascending chains of principal ideals of $M$ that do not stabilize. Because $M$ is a Puiseux monoid that does not satisfy the ACCP, it follows from Corollary~\ref{cor:ACCP equivalent to strong MCD for PMs} that $M$ is not strongly MCD. However, $M$ is an MCD monoid as a consequence of Theorem~\ref{thm:2-prime reciprocal MCD}, which we will state and prove later. Thus, similar to Grams' monoid, $M$ is an atomic MCD monoid that is not strongly MCD or, equivalently, that does not satisfy the ACCP. The atomic structure of the monoid $M$ has also been considered in recent literature (see, for instance, \cite[Section~3]{GL23a}).

    Finally, it is worth emphasizing that (as with Grams' monoid) whether the monoid algebra $F[M]$ is an atomic integral domain for any field $F$ is still an open problem.
\end{example}

\bigskip
\section{A Class of Grams-Like Puiseux Monoids}
\label{sec:generalization of Grams' monoid}

Motivated by Grams' monoid, we proceed to introduce a construction on a class of Puiseux monoids that respects the existence and uniqueness of maximal common divisors.

Let $(d_n)_{n \ge 1}$ be a strictly increasing sequence of positive integers, and let $(p_n)_{n \ge 1}$ be a sequence of pairwise distinct primes such that $p_n \nmid d_m$ for any $m,n \in \nn$. Now consider the following Puiseux monoids:
\begin{equation} \label{eq:Grams generalization}
	M := \Big\langle \frac{1}{d_np_n} : n \in \mathbb{N} \Big\rangle \quad \text{ and } \quad N := \Big\langle \frac{1}{d_n} : n \in \mathbb{N} \Big\rangle.
\end{equation}
We call $M$ the \emph{Grams-like monoid} of the sequences $(d_n)_{n \ge 1}$ and $(p_n)_{n \ge 1}$ or, simply, a \emph{Grams-like monoid}. The following basic properties of Grams-like monoids generalize those satisfied by Grams' monoid.

\begin{prop} \label{prop:<1/(x_np_n)> is not ACCP}
	Let $M$ and $N$ be the Puiseux monoids in~\eqref{eq:Grams generalization}. Then the following statements hold.
	\begin{enumerate}
		\item $M$ is atomic with $\mathcal{A}(M) = \big\{ \frac{1}{d_np_n} : n \in \mathbb{N} \big\}$.
		\smallskip
		
		\item If $N$ is a valuation monoid, then $M$ does not satisfy the ACCP.
		
	\end{enumerate}
\end{prop}

\begin{proof}
    (1) Observe that, for every $n \in \mathbb{N}$, the element $\frac{1}{d_n p_n}$ is the only defining generator of $M$ with negative $p_n$-adic valuation. This implies that the set of atoms of the monoid $M$ is 
    $
    \mathcal{A}(M) = \{ \frac{1}{d_n p_n} : n \in \mathbb{N} \},
    $
    and therefore $M$ is atomic.
    
	\smallskip

	(2) Since $N$ is a valuation Puiseux monoid, for each $n \in \nn$, the fact that $d_n < d_{n+1}$ guarantees that $\frac{1}{d_{n+1}} \mid_N \frac{1}{d_n}$, whence $\frac{1}{d_{n+1}} \mid_M \frac{1}{d_n}$. Therefore $\big(\frac1{d_n} + M\big)_{n \ge 1}$
    is an ascending chain of principal ideals of $M$. Moreover, for each $n \in \nn$, the inequality $\frac{1}{d_{n+1}} < \frac{1}{d_n}$ guarantees that $\frac{1}{d_{n+1}} \notin \frac{1}{d_n} + M$, which implies that $\frac{1}{d_n} + M$ is strictly contained in $\frac{1}{d_{n+1}} + M$. As a consequence, one obtains that the ascending chain $\big(\frac1{d_n} + M\big)_{n \ge 1}$ of principal ideals of $M$ does not stabilize. Hence $M$ does not satisfy the ACCP, as desired.
	
\end{proof}

In order to study aspects concerning to the divisibility of Grams-like monoids, the existence of certain canonical sum decomposition plays a crucial role.

\begin{lemma} \label{lem:rep of elements of M}
	Let $M$ and $N$ be the Puiseux monoids in~\eqref{eq:Grams generalization}. Then each $q \in M$ can be uniquely written as follows:
    \begin{equation} \label{eq:canonical decomp of generalized GM}
        q = c_0(q) + \sum_{n \in \nn} c_n(q)\frac1{d_np_n},
    \end{equation}
    where $c_0(q) \in N$ and $(c_n(q))_{\ge 1}$ is a sequence that eventually stabilizes to $0$ such that $c_n(q) \in \ldb 0, p_n-1\rdb$ for every $n \in \nn$.
\end{lemma}

\begin{proof}
    Fix $q \in M$. If $q = 0$, then it is clear that $q$ has a unique sum decomposition as in~\eqref{eq:canonical decomp of generalized GM}, namely, the sum decomposition with $c_0(q) = 0$ and the sequence $(c_0(q))_{n \ge 1}$ all whose terms are zero. Thus, we assume that $q \neq 0$. 
    
    We first argue the existence of a sum decomposition for $q$ as that specified in~\eqref{eq:canonical decomp of generalized GM}. Since the monoid $M$ is atomic, we can take a sequence $(b_n)_{n \ge 1}$ of nonnegative integers that eventually stabilizes to $0$ such that the equality $q = \sum_{n \in \nn} b_n\frac{1}{d_np_n}$ holds. For each $n \in \nn$, we can set $c_n(q) := b_n - \lfloor \frac{b_n}{p_n} \rfloor p_n$, and observe that $c_n(q) \in \ldb 0, p_n-1 \rdb$ because it is the remainder of $b_n$ modulo $p_n$. Now we see that the identity
    \[
        q = \sum_{n \in \nn} \bigg\lfloor\frac{b_n}{p_n}\bigg\rfloor \frac{1}{d_n} + \sum_{n \in \nn} c_n(q)\frac1{d_np_n}
    \]
    yields a sum decomposition of $q$ as that in~\eqref{eq:canonical decomp of generalized GM}: indeed, just set $c_0(q) := \sum_{n \in \nn} \lfloor\frac{b_n}{p_n}\rfloor \frac{1}{d_n}$, which is a finite sum that belongs to $N$ (because the sequence $(b_n)_{n \ge 1}$ has only finitely many nonzero terms).
    
    We show that the sum decomposition of $q$ given in~\eqref{eq:canonical decomp of generalized GM} must be unique. Towards this end, suppose that
    \[
        q = c'_0(q) + \sum_{n \in \nn} c'_n(q)\frac1{d_np_n},
    \]
    is also a sum decomposition of $q$ as in~\eqref{eq:canonical decomp of generalized GM}, that is, $c'_0(q) \in N$ and $(c'_n(q))_{\ge 1}$ is a sequence that eventually stabilizes to~$0$ with $c'_n(q) \in \ldb 0, p_n-1\rdb$ for every $n \in \nn$. As both sum decompositions equal $q$, the following equality holds:
    \begin{equation} \label{eq:equality existence}
         (c'_0(q) - c_0(q)) = \sum_{n \in \nn} (c_n(q) - c'_n(q)) \frac1{d_n p_n}.
    \end{equation}
    For each $n \in \nn$, note that the left-hand side of~\eqref{eq:equality existence} has nonnegative $p_n$-adic valuation because $p_n \nmid d_m$ for any $m \in \nn$, so we can apply the $p_n$-adic valuation map on both sides of~\eqref{eq:equality existence} to obtain that $c_n(q) = c'_n(q)$. This implies that $c_0(q) = c'_0(q)$. Hence the two sum decompositions of $q$ are the same, completing the proof of our claim.
\end{proof}

The sum decompositions described in Lemma~\ref{lem:rep of elements of M} are crucial in the rest of this section as we study Grams-like monoids and their maximal common divisors. Thus, we introduce the following terminology.

\begin{definition}
    Let $M$ be the Grams-like monoid in~\eqref{eq:Grams generalization}. For each $q \in M$, we call
    \begin{equation} \label{eq:canonical decomp of generalized GM }
        c_0(q) + \sum_{n \in \nn} c_n(q)\frac1{d_np_n}
    \end{equation}
    the \emph{canonical sum decomposition} of $q$ provided that the element $c_0(q)$ and the sequence $(c_n(q))_{n \ge 1}$ satisfy the conditions specified in Lemma~\ref{lem:rep of elements of M}.
\end{definition}

Let us take a look at some immediate consequences of the uniqueness of the canonical sum decomposition.

\begin{cor} \label{cor:e_1|e_2}
	Let $M$ and $N$ be the Puiseux monoids in~\eqref{eq:Grams generalization}. For any $q_1, q_2 \in M$ such that $q_1 \mid_M q_2$, let
     \[
        q_1 = c_0(q_1) + \sum_{n \in \nn} c_n(q_1)\frac1{d_np_n} \quad \text{and} \quad q_2 = c_0(q_2) + \sum_{n \in \nn} c_n(q_2)\frac1{d_np_n}
     \]
    be the canonical sum decompositions of $q_1$ and $q_2$ described in~\eqref{eq:canonical decomp of generalized GM }, respectively. Then the following statements hold.
	\begin{enumerate}
		\item $c_0(q_1) \mid_N c_0(q_2)$.
        \smallskip
        
		\item $c_n(q_2) = (c_n(q_1) + c_n(q_2 - q_1)) \pmod{p_n}$ for every $n \in \mathbb{N}$.
        \smallskip
        
		\item If $c_0(q_1) = c_0(q_2)$, then $c_0(q_2 - q_1) = 0$.
        \smallskip
        
		\item If $c_0(q_1) = c_0(q_2)$, then $c_n(q_1) + c_n(q_2 - q_1) < p_n$ for every $n \in \mathbb{N}$.
        \smallskip
        
		\item If $c_0(q_1) = c_0(q_2)$, then $c_n(q_1) \le c_n(q_2)$ for every $n \in \mathbb{N}$.
	\end{enumerate}
\end{cor}

\begin{proof}
	Write $q_2 = q_1 + q_3$ for some $q_3 \in M$,  let $r_n$ be the remainder of $c_n(q_1) + c_n(q_3)$ modulo $p_n$ for every $n \in \nn$. Observe that
    \[
        q_2 = \bigg( c_0(q_1) + c_0(q_3) + \sum_{n \in \mathbb{N} \, : \, c_n(q_1) + c_n(q_3) \ge p_n} \frac{1}{d_n}\bigg) + \sum_{n \in \nn} r_n \frac1{d_n p_n}
    \]
    is the canonical sum decomposition of $q_2$ in $M$, and so it follows from the uniqueness of such a canonical sum decomposition that
    \[
        c_0(q_2) = c_0(q_1) + c_0(q_3) + \sum_{n \in \mathbb{N} \, : \, c_n(q_1) + c_n(q_3) \ge p_n} \frac1{d_n}.
    \]
    All parts of the corollary are now immediate, with part~(5) being a direct consequence of part~(4).
\end{proof}

With notation as in~\eqref{eq:Grams generalization}, we are in a position to prove that $M$ is an MCD monoid provided that $N$ is a valuation monoid.

\begin{theorem} \label{thm:when Grams-like monoids are MCD}
	Let $M$ and $N$ be the Puiseux monoids in~\eqref{eq:Grams generalization}. If $N$ is a valuation monoid, then $M$ is an MCD monoid.
\end{theorem}

\begin{proof}
	Assume that $N$ is a valuation monoid. Let $S$ be a nonempty finite subset of~$M$, and let us prove that $S$ has a maximal common divisor in $M$. Since $N$ is a valuation Puiseux monoid, after replacing $S$ by $\{s-c_0 : s \in S\}$, where $c_0 := \min \{c_0(s) : s \in S\}$, we can assume that
    \begin{equation} \label{eq:aux min of a set}
      \min \{c_0(s) : s \in S\} = 0.
    \end{equation}
    Now take the largest possible index $k \in \nn$ such that there exists an element $s \in S$ with $c_k(s) \neq 0$. Consider the finitely generated submonoid
    \[
        F := \Big\langle \frac1{d_ip_i} : i \in \ldb 1,k \rdb \Big\rangle
    \]
    of $M$, and let $m$ be the maximum element of $F$ that is a common divisor of $S$ in $M$ (observe that such a maximum element~$m$ must exist because $F$ is finitely generated and so for any $r \in \rr_{\ge 0}$, the monoid $F$ has only finitely many elements contained in the real closed interval $[0,r]$). After replacing $S$ by the set $\{s-m : s \in S\}$ if needed, one can further assume that $0$ is the only element of $F$ that is a common divisor of $S$ in $M$. 
    
    In light of~\eqref{eq:aux min of a set}, we can take $q \in S$ such that $c_0(q) = 0$. If $p_n \mid \mathsf{d}(q)$ for some $n \in \nn$, then $n \in \ldb 1,k \rdb$: indeed, if $p_n \mid \mathsf{d}(q)$, then the coefficient $c_n(q)$ in the canonical sum decomposition of $q$ in $M$ is nonzero (because $p_n \nmid d_m$ for any $m \in \nn$), and so the fact that $n \in \ldb 1,k \rdb$ follows from the maximality of $k$ as $q \in S$. We proceed to establish the following claim as it is needed to complete the proof.
    \smallskip
    
    \noindent \textsc{Claim.} If $a \mid_M q$ for some $a \in \mathcal{A}(M)$, then $a \in \big\{\frac{1}{d_i p_i} : i \in \ldb 1,k \rdb \big\}$.
    \smallskip
    
    \noindent \textsc{Proof of Claim.} Suppose, by way of contradiction, that there exists an index $\ell \in \nn$ with $\ell > k$ such that $\frac{1}{d_\ell p_\ell} \mid_M q$. Now let
    \begin{equation}\label{eq:canonical sum dec of q for a claim}
        q = c_0(q) + \sum_{n=1}^k c_n(q)\frac1{d_n p_n}
    \end{equation}
    be the canonical sum decomposition of $q$ in $M$. Since $\ell > k$, it follows from the argument given in the paragraph preceding the statement of the claim that $p_\ell \nmid \mathsf{d}(q)$, and so $p_\ell \mid \mathsf{d}\big(q-\frac1{d_\ell p_\ell}\big)$, which in turn implies that
    \[
        c_\ell := c_\ell\big(q-\frac1{d_\ell p_\ell}\big) \neq 0
    \]
    in the canonical sum decomposition of $q-\frac1{d_\ell p_\ell}$. In addition, as $\ell$ is the largest positive integer such that $p_\ell \mid \mathsf{d}\big(q - \frac1{d_\ell p_\ell}\big)$, we obtain that $c_\ell$ is the largest nonzero coefficient in the canonical sum decomposition of $q - \frac1{d_\ell p_\ell}$, and so we can rewrite the canonical sum decomposition of $q - \frac1{d_\ell p_\ell}$ as follows:
    \begin{equation} \label{eq:another sum for q}
        q = \frac{1}{d_\ell p_\ell} + c_0\Big(q - \frac1{d_\ell p_\ell}\Big) + \Big(\sum_{n=1}^{\ell-1} c_n\Big(q - \frac1{d_\ell p_\ell}\Big) \frac1{d_n p_n}\Big) + \frac{c_\ell}{d_\ell p_\ell}.
    \end{equation}
    Now we can apply the $p_\ell$-adic valuation map on both sides of~\eqref{eq:another sum for q}, and use the fact that $p_\ell \nmid \mathsf{d}(q)$ to see that $p_\ell = c_\ell + 1$, in which case, \eqref{eq:another sum for q} would be a canonical sum decomposition of $q$ different from that in~\eqref{eq:canonical sum dec of q for a claim}, contradicting the uniqueness of the canonical sum decomposition of $q$ in $M$. Hence the claim is established.
    \smallskip
    
    We are in a position to finish the proof. First observe that the existence of a nonzero common divisor of $S$ in $M$ is equivalent to the existence of a common divisor $a \in \mathcal{A}(M)$ of $S$ in $M$, in which case, the divisibility relation $a \mid_M q$ would imply that $a \in F$ (as a consequence of our established claim), which cannot be possible because $0$ is the only element of $F$ that is a common divisor of $S$ in $M$.   
\end{proof}

Next, we want to characterize when a Grams-like monoid is MCD-finite. First, we need the following lemma on common divisors.

\begin{lemma} \label{lem:g_k(mcd) is 0 if g_k(s) is 0 for all s in S}
	Let $M$ be the Puiseux monoid in~\eqref{eq:Grams generalization}, and let $S$ be a nonempty subset of~$M$. If $d$ is a common divisor of $S$ in $M$ and $n$ is a positive integer such that $c_n(s) = 0$ for all $s \in S$ and $c_n(d) > 0$, then $d + \frac{p_n-c_n(d)}{d_n p_n}$ is also a common divisor of $S$ in $M$.
\end{lemma}

\begin{proof}
	Let $d$ be a common divisor of $S$ in $M$, and take $n \in \nn$ such that $c_n(s) = 0$ for all $s \in S$ and $c_n(d) > 0$. To argue that $d + \frac{p_n-c_n(d)}{d_n p_n}$ is a common divisor of $S$ in $M$, fix $s \in S$. By part (2) of Corollary~\ref{cor:e_1|e_2}, we know that $p_n$ divides $(c_n(d) + c_n(s-d)) - c_n(s)$, which implies that $c_n(d) + c_n(s-d) = p_n$ because $c_n(s) = 0$ and $c_n(d) > 0$. Hence the equality $\frac{p_n - c_n(d)}{d_n p_n} = c_n(s-d)\frac1{d_n p_n}$ holds, and so $\frac{p_n - c_n(d)}{d_n p_n}$ is a summand in the canonical sum decomposition of $s-d$ in $M$. As a result, $\frac{p_n - c_n(d)}{d_n p_n} \mid_M s-d$, which implies that $d + \frac{p_n - c_n(d)}{d_n p_n} \mid_M s$, as desired.
\end{proof}

We conclude this section characterizing the Grams-like monoids that are MCD-finite.

\begin{theorem}
	Let $M$ and $N$ be the Puiseux monoids in~\eqref{eq:Grams generalization}. Then $M$ is MCD-finite if and only if $N$ is MCD-finite.
\end{theorem}

\begin{proof}
    For the direct implication, suppose that $M$ is MCD-finite. Fix a nonempty finite subset $S$ of $N$, and let us argue that $S$ has finitely many MCDs in $N$. As the monoid $M$ is MCD-finite, it suffices to prove that each MCD of $S$ in $N$ is also an MCD of $S$ in $M$. Assume, by way of contradiction, that there exists $d \in N$ that is an MCD of $S$ in $N$ but not an MCD of $S$ in $M$. Then $d$ is a non-maximal common divisor of $S$ in $M$, and so we can take a nonzero $q \in M$ such that $d+q$ is a common divisor of $S$ in $M$. Thus, for each $s \in S$, it follows from part~(1) of Corollary~\ref{cor:e_1|e_2} that $c_0(d+q) \mid_N c_0(s)$, which we can rewrite as $d + c_0(q) \mid_N s$ because $d,s \in N$. Hence $d + c_0(q)$ is a common divisor of $S$ in $N$, and this implies that $c_0(q) = 0$ because $d$ is an MCD of $S$ in $N$.
    
    As $q$ is positive, there must exist an index $n \in \nn$ such that $c_n(q) > 0$, whence $c_n(d+q) = c_n(q) > 0$ as $d \in N$. 
    Now, since $c_n(s) = 0$ for all $s \in S$ and $d+q$ is a common divisor of $S$ in $M$ with $c_n(d+q) > 0$, Lemma~\ref{lem:g_k(mcd) is 0 if g_k(s) is 0 for all s in S} guarantees that $(d+q) + \frac{p_n - c_n(d+q)}{d_n p_n}$ is also a common divisor of $S$ in $M$. Hence using the equality $c_n(d+q) = c_n(q)$, we obtain the following conclusion:
    \begin{equation} \label{eq:divisibility relation in MCD-finite}
        \Big(d + \frac1{d_n} \Big) + q \, \big{|}_M \, s + c_n(q)\frac1{d_n p_n}
    \end{equation}
    for all $s \in S$. Since $d + \frac1{d_n} \in N$ and $c_0(q) = 0$, the equality $c_0\big(d + \frac1{d_n} + q\big) = d + \frac1{d_n}$ holds. Also, the inclusion $S \subseteq N$ implies that $c_0\big(s + c_n(q)\frac1{d_n p_n}\big) = s$. Now we can use part~(1) of Corollary~\ref{cor:e_1|e_2} to infer from~\eqref{eq:divisibility relation in MCD-finite} that $d + \frac1{d_n} \mid_N s$ for all $s \in S$, which means that $d + \frac1{d_n}$ is a common divisor of $S$ in $N$. However, this contradicts that $d$ is an MCD of $S$ in $M$.    
    \smallskip

	For the reverse implication, suppose that $N$ is MCD-finite. Let $S$ be a nonempty finite subset of $M$, and let us prove that the set $\text{MCD}_M(S)$, which consists of all MCDs of $S$ in $M$, is finite. The following nonempty finite subset of $N$ will be helpful:
    \[
        S_0 := \{c_0(s) : s \in S\}.
    \]
    It immediately follows from part (1) of Corollary~\ref{cor:e_1|e_2} that for any $d \in \text{MCD}_M(S)$, the element $c_0(d)$ is an MCD of $S_0$ in $N$. As $N$ is MCD-finite, $S_0$ has only finitely many MCDs in $N$. Thus, it suffices to argue that
    \[
        D := \{d \in \text{MCD}_M(S) : c_0(d) = d_0\}
    \]
    is a finite subset of $M$, assuming that $d_0$ is an arbitrarily taken MCD of $S_0$ in $N$. Since $S$ is finite, we can take $k \in \nn$ to be the maximum positive index such that $c_k(s) > 0$ for some $s \in S$. Let us establish the following claim.
    \smallskip

    \noindent \textsc{Claim.} If $(d, \ell) \in D \times \nn_{> k}$, then $c_\ell(d) = 0$ holds.
    \smallskip

    \noindent \textsc{Proof of Claim.} Assume, by way of contradiction, that there exist $d \in D$ and $\ell \in \nn$ such that $\ell > k$ and $c_\ell(d) > 0$. As $\ell > k$, the maximality of $k$ guarantees that $c_\ell(s) = 0$ for all $s \in S$. In addition, we know that $d$ is a (maximal) common divisor of $S$ in $M$ with $c_\ell(d) > 0$. We have just checked the conditions needed in Lemma~\ref{lem:g_k(mcd) is 0 if g_k(s) is 0 for all s in S} to obtain that $d + \frac{p_\ell - c_\ell(d)}{d_\ell p_\ell}$ is a common divisor of $S$ in $M$. However, this contradicts the maximality of $d$ as a common divisor of $S$ in $M$.
    \smallskip

    In light of the established claim, the canonical sum decomposition of each element $d \in D$ can be written as $d_0 + \sum_{i=1}^k c_i(d) \frac1{d_i p_i}$, and there are only finitely many of such canonical sum decompositions, namely, $\prod_{i=1}^k p_i$. As a consequence, the set $D$ is finite, and so is $\text{MCD}_M(S)$. Hence $M$ is MCD-finite.
\end{proof}

\bigskip
\section{A Class of Weak Reciprocal Puiseux Monoids}
\label{sec:2-prime reciprocal monoids}

Motivated by the Puiseux monoid in Example~\ref{ex:prototypical 2-prime reciprocal}, we introduce a class consisting of atomic Puiseux monoids not satisfying the ACCP, and then we study certain atomic decompositions and the existence of maximal common divisors inside monoids in such class.
\smallskip

Let $(p_n)_{n\geq1}$ be a strictly increasing sequence of primes. We call the Puiseux monoid defined as
\[
    M := \left\langle\frac1{p_np_{n+2}}\colon n \in \mathbb N\right\rangle
\]
the $2$-\emph{prime reciprocal} Puiseux monoid of the sequence $(p_n)_{n \ge 1}$ or, simply, a $2$-\emph{prime reciprocal} Puiseux monoid. We will denote the $2$-prime reciprocal Puiseux monoid of the strictly increasing sequence of primes by $M_\pp$. It is known that $M_\pp$ is strongly atomic (i.e., atomic and $2$-MCD) but does not satisfy the ACCP. As we proceed to show, the same holds for any $2$-prime reciprocal Puiseux monoid.

\begin{prop} \label{prop:2-prime reciprocal is atomic}
    Let $(p_n)_{n \ge 1}$ be a strictly increasing sequence of primes, and let $M$ be the $2$-prime reciprocal Puiseux monoid of the sequence $(p_n)_{n \ge 1}$. Then the following statements hold.
    \begin{enumerate}
        \item $M$ is atomic with $\mathcal{A}(M) = \big\{ \frac{1}{p_n p_{n+2}} : n \in \nn \big\}$.
        \smallskip

        \item $M$ does not satisfy the ACCP.
    \end{enumerate}
\end{prop}

\begin{proof}
    (1) It suffices to show that $\frac1{p_n p_{n+2}}$ is an atom of $M$ for every $n \in \nn$. Fix $i \in \nn$. If $\frac1{p_ip_{i+2}}$ can be written as the sum of other elements of $A$, then at least one of them must have denominator divisible by $p_i$, which means that at least one element in the sum is either $\frac1{p_{i-2}p_i}$ or $\frac1{p_ip_{i+2}}$, which are both impossible since both numbers are at least $\frac1{p_ip_{i+2}}$. Therefore $\frac1{p_ip_{i+2}}$ is an atom of $M$. 
    \smallskip

    (2) From the fact that $\frac1{p_n} - \frac1{p_{n+2}} = \frac{p_{n+2}-p_n}{p_n p_{n+2}}\in M$ for every $n\in\mathbb N$, we obtain that the ascending chains $\big( \frac1{p_{2n-1}} + M \big)_{n \ge 1}$ and $\big( \frac1{p_{2n}} + M \big)_{n \ge 1}$ of principal ideals of $M$ do not stabilize (this is exactly as in Example~\ref{ex:prototypical 2-prime reciprocal}). Hence the monoid $M$ does not satisfy the ACCP.
\end{proof}

In~\cite[Example~3.9 and Proposition~3.10]{GL23a}, Gotti and Li show that $M_\pp$ is atomic and $2$-MCD. We not only generalize this result to all $2$-prime reciprocal Puiseux monoids, but we also take a step forward proving that all $2$-prime reciprocal Puiseux monoids are MCD monoids. First, we need to argue the existence of a certain sum decomposition inside any $2$-prime reciprocal monoid. 

\begin{lemma} \label{lem:2-prime reciprocal decomposition}
	Let $(p_n)_{n \ge 1}$ be a strictly increasing sequence of primes, and let $M$ be the $2$-prime reciprocal Puiseux monoid induced by $(p_n)_{n \ge 1}$. Then each $q\in M$ can be written as follows:
    \[
        q=c+\sum_{i=-1}^{n-2} c_{i+2}\frac{1}{p_ip_{i+2}},
    \]
    where $n = \max(\{0\} \cup \{i \in \nn :v_{p_i}(q)<0\})$, $p_{-1}=p_0=1$, 
    and $c,c_i \in \mathbb N_0$ for every $i \in \ldb 1, n \rdb$.
\end{lemma}

\begin{proof}
	If $n\leq1$, then the result is obvious. Therefore we assume that $n\geq2$. It follows from part~(1) of Proposition~\ref{prop:2-prime reciprocal is atomic} that $M$ is an atomic with $\mathcal{A}(M) = \big\{ \frac{1}{p_np_{n+2}} : n \in \nn \big\}$. Hence we can write $q$ in $M$ as follows:
	\[
		q=c+\sum_{i=-1}^m \frac{c_{i+2}}{p_ip_{i+2}},
	\]
	where $m\in\mathbb N$ and $c, c_1, c_2, \dots, c_{m+2} \in \nn_0$. We further assume that we have taken $m$ as small as it can possibly be, and we proceed to argue that $m \le n-2$. Assume, towards a contradiction, that $m\geq n-1\geq 1$. Then the only term with a denominator that could be divisible by $p_{m+2}$ is $\frac{c_{m+2}}{p_mp_{m+2}}$, so we must have $p_{m+2}\mid c_{m+2}$ since $p_{m+2}$ does not divide the denominator of $a$. This implies we can replace $c_m$ with $c_m+p_{m-2}\frac{c_{m+2}}{p_{m+2}}$, remove $c_{m+2}$, then subtract $1$ from $m$ since $$\frac{c_m+p_{m-2}\frac{c_{m+2}}{p_{m+2}}}{p_{m-2}p_m}=\frac{c_m}{p_{m-2}p_m}+\frac{c_{m+2}}{p_mp_{m+2}}.$$ This reduces the value of $m$, contradicting its minimality. 
\end{proof}

Now we can prove the main result of this section.

\begin{theorem} \label{thm:2-prime reciprocal MCD}
    Any 2-prime reciprocal Puiseux monoid is an MCD monoid.
\end{theorem}

\begin{proof}
    Let $(p_n)_{n\ge 1}$ be a strictly increasing sequence of primes, and let $M$ be the $2$-prime reciprocal Puiseux monoid induced by $(p_n)_{n \ge 1}$. For any $q \in M$, it follows from  Lemma~\ref{lem:2-prime reciprocal decomposition} that if $p_n$ is the largest prime dividing the denominator of $q$, then there exist $b, b_1, b_2, \dots, b_n \in \nn_0$ such that
	\[
		q = b + \sum_{i=-1}^{n-2} \frac{b_{i+2}}{p_ip_{i+2}}.
	\] 
    In particular, if $n \ge 3$, then the only term with denominator divisible by $p_n$ is $\frac{b_n}{p_{n-2}p_n}$, so $b_n$ must be nonzero, implying $q$ must be divisible by $\frac1{p_{n-2}p_n}$.
	
	Let $S = \{s_1, \dots, s_k\}$ be a nonempty subset of $M$. We want to show that $S$ has an MCD in $M$. As $\nn_0 \subset M$, if every element of $S$ is an integer, then $\min S$ is an MCD of $S$ in $M$. Otherwise, let $p_{n-5}$ be the largest prime dividing the denominator of any element of $S$. There exist $n_1, \dots, n_k \in \nn_0$ and sequences $(c_{1,j})_{j \ge 1}, \dots, (c_{k,j})_{j \ge 1}$ of coefficients with terms in $\nn_0$ such that
	\[
		s_i=n_i+\sum_{j=-1}^n\frac{c_{i,j+2}}{p_jp_{j+2}}
    \]
	for every $i \in \ldb 1,k \rdb$. Let $d$ be the largest common divisor of $S$ of the form 
        \[d=d_0+\sum_{i=-1}^n\frac{d_{i+2}}{p_ip_{i+2}},\]
    which exists because there are finitely many possible values for each of $d_1$, $d_2$, \dots, $d_{n+2}$ when $d\leq\min S$. We are done once we establish the following claim.
    \smallskip

    \noindent \textsc{Claim.} $d$ is an MCD of $S$ in $M$.
    \smallskip

    \noindent \textsc{Proof of Claim.} Assume, by way of contradiction, that $d$ is not an MCD of $S$. Given our assumption, there exists a positive $e \in M$ that is a common divisor of $S-d$ in $M$. Let $p_m$ be the largest prime dividing the denominator of $e$. Assume that we have picked $e$ such that $m$ is minimal. Now, fixing this minimal value of $m$, we replace $e$ by the largest possible positive common divisor of $S-d$ in $M$ such that the largest prime dividing its denominator is $p_m$ (which exists since each $m$ only has a finite number of possible $e$). Since $m>n+2$, for any $i \in \ldb 1,k \rdb$, the largest prime dividing the denominator of $s_i-d-e$ is $p_m$, so $s_i-d-e$ is  divisible by $\frac1{p_{m-2}p_m}$, so $e+\frac1{p_{m-2}p_m}$ is also a common divisor of $S-d$. However, this contradicts either the maximality of $e$ or the minimality of $m$, and so the claim is established. This concludes the proof. 
\end{proof}

\bigskip
\section{Atomic Monoids that Are not MCD Monoids}
\label{sec:atomic not strongly atomic}

The atomic monoids we have considered so far are all MCD monoids. In this last section, we provide a new construction of an atomic monoid that is not MCD (not even $2$-MCD). Thus, the monoid resulting from our construction will not satisfy neither the strong MCD property nor the ACCP. In addition, the same monoid will be produced as a maximal-rank submonoid of the additive monoid $\qq^2_{\ge 0}$, and so it will have rank~$2$. Working with submonoids of $\mathbb Q_{\ge 0}^2$ has advantages similar to working with Puiseux monoids, including that every submonoid of $\qq_{\ge 0}^2$ is reduced and torsion-free. In our construction, we will need the projection homomorphisms $\pi_1, \pi_2 \colon \mathbb Q_{\ge 0}^2 \to \mathbb Q_{\ge 0}$, which are defined as follows: for any $(x,y) \in \qq^2_{\ge 0}$,
\begin{align*}
	\pi_1((x,y))&:=x\\
	\pi_2((x,y))&:=y.
\end{align*}

\begin{example}
	Let $p_n$ denote the $n$-th prime number. Consider the submonoid $M$ of $\mathbb Q_{\ge 0}^2$ defined in the following way:
	\[
		M: = \left<\left(\frac{1}{2^np_{2n}},0\right),\left(\frac{1}{2^np_{2n+1}},\frac{1}{p_{2n+1}}\right) : n \in \mathbb N \right>.
	\]
	Under applying $\pi_1$, the resulting generating set is $S :=\big\{ \frac{1}{2^np_{2n}}, \frac{1}{2^np_{2n+1}} : n \in \mathbb N \big\}$. Note that, for each $n \in \nn$, it is impossible to express $\frac{1}{2^np_{2n}}$ as the sum of other elements of $S$ because such a sum cannot have a denominator divisible by $p_{2n}$, while a similar situation happens with $\frac{1}{2^np_{2n+1}}$. Since $\pi_1$ is a monoid homomorphism, this means every element of the defining generating set of $M$ must be an atom. So $M$ is an atomic monoid.
	
	Furthermore, note that $M$ contains the elements $\left(\frac{1}{2^n},0\right)$ and $\left(\frac{1}{2^n},1\right)$ for any positive integer $n$. Thus $M$ contains the elements $(d,0)$ and $(d,1)$ for any positive dyadic rational $d$. Note that $(0,1) \notin M$ because every generator of $M$ is positive under $\pi_1$ while $\pi_1((0,1))=0$. Next, we prove the following claim.
    \smallskip
    
    \noindent \textsc{Claim.} Each common divisor in $M$ of the set $\{(1,0),(1,1)\}$ divides an element of the form $\left(1-\frac{1}{2^n},0\right)$ for some positive integer $n$. 
    \smallskip
	
	\noindent \textsc{Proof of Claim.} Let $x$ be a common divisor of the set $\{ (1,0), (1,1) \}$ in $M$. Then we have $(1,0) = x+a$ and $(1,1) = x+b$ for some $a,b\in M$. Since $\mathbb Q_{\ge 0}^2$ is cancellative, this implies $b=a+(0,1)$. Furthermore, since $\pi_2((1,0))=0$, this means $\pi_2(a)=\pi_2(x)=0$ and therefore $\pi_2(b)=1$.
	
	Consider an atomic decomposition of $a$ and $b$ in $M$. Since $\pi_2(a)=0$, the atomic decomposition of $a$ only uses elements of the form $\big(\frac{1}{2^np_{2n}},0 \big)$ for positive integers $n$. Therefore the denominator of $\pi_1(a)$ cannot be divisible by $p_{2n+1}$ for positive integers $n$. Consequently, the same is true of $\pi_1(b)$. This means that in the atomic decomposition of $b$, an atom of the form $\big( \frac{1}{2^np_{2n+1}},\frac{1}{p_{2n+1}} \big)$ is used a multiple of $p_{2n+1}$ times for any positive integer $n$. However,
	\[
		\pi_2 \bigg( p_{2n+1}\bigg(\frac{1}{2^np_{2n+1}}, \frac{1}{p_{2n+1}} \bigg)\bigg) = 1 = \pi_2(b).
	\]
	Therefore there must be a unique $n_0 \in \nn$ such that $\big(\frac{1}{2^{n_0}p_{2n_0+1}},\frac{1}{p_{2n_0+1}} \big)$ appears in the atomic decomposition of $b$ and it appears exactly $p_{2n_0+1}$ times. 
    Then $\big(\frac{1}{2^{n_0}},1\big)$ divides $b$. Hence $x=(1,1)-b$ divides $\big(1-\frac{1}{2^{n_0}}, 0\big)$. The claim is now established.
    \smallskip

    We are now ready to show that $M$ is not a 2-MCD monoid. Then for any common divisor $d$ of $(1,0)$ and $(1,1)$, we can choose some positive integer $n$ such that $d \mid_M 1-\frac{1}{2^n}$ and $d\neq 1-\frac{1}{2^n}$. If $d$ is of the form $1-\frac{1}{2^k}$ for some positive integer $k$, then we can pick $n=k+1$. Then $1-\frac{1}{2^n}-d\neq 0$, so it is not invertible in $M$. We further have $1-\frac{1}{2^n}$ is a common divisor of $(1,0)$ and $(1,1)$. Therefore $M$ is not a 2-MCD monoid.
\end{example}

\bigskip
\section*{Acknowledgments}

We would like to thank our mentor Felix Gotti for guiding us during our research. We also thank the MIT PRIMES program for making this collaboration possible.

\bigskip
\section*{Conflict of Interest Declaration}

Authors have no conflict of interest to declare.

\bigskip


\begin{thebibliography}{20}

    \bibitem{AGH21} C. Aguilera, M. Gotti, A.~F. Hamelberg, \emph{Factorizations in reciprocal Puiseux monoids}. Preprint on arXiv: https://arxiv.org/abs/2112.04048

	\bibitem{AAZ90} D.~D. Anderson, D.~F. Anderson, and M. Zafrullah, \emph{Factorizations in integral domains}, J. Pure Appl. Algebra \textbf{69} (1990) 1--19.

	\bibitem{BBNS23} J. P. Bell, K. Brown, Z. Nazemian, and D. Smertnig, \emph{On noncommutative bounded factorization domains and prime rings}, J. Algebra \textbf{622} (2023) 404--449.
	
	\bibitem{BC19} J. G. Boynton and J. Coykendall, \emph{An example of an atomic pullback without the ACCP}, J. Pure Appl. Algebra \textbf{223} (2019) 619--625.

    \bibitem{BGLZ24} A. Bu, F. Gotti, B. Li, and A. Zhao, \emph{One-dimensional monoid algebras and ascending chains of principal ideals}. Submitted. Preprint on arXiv: https://arxiv.org/abs/2409.00580
		
	\bibitem{CGG20a} S.~T. Chapman, F. Gotti, and M. Gotti, \emph{When is a Puiseux monoid atomic?}, Amer. Math. Monthly \textbf{128} (2021) 302--321.

    \bibitem{CJMM24} S. T. Chapman, J. Jang, J. Mao, and S. Mao, \emph{On the set of Betti elements of a Puiseux monoid}, Bull. Aust. Math. Soc. (to appear). DOI: https://doi.org/10.1017/S0004972724000352

	\bibitem{pC68} P. M. Cohn, \emph{Bezout rings and their subrings}, Proc. Cambridge Philos. Soc. \textbf{64} (1968) 251--264.

    \bibitem{CG24} J. Coykendall and F. Gotti, \emph{Atomicity in integral domains}. In: Recent Progress in Rings and Factorization Theory, Springer, 2025.

	\bibitem{CG19} J. Coykendall and F. Gotti, \emph{On the atomicity of monoid algebras}, J. Algebra \textbf{539} (2019) 138--151.

	\bibitem{EK18} S. Eftekhari and M. R. Khorsandi, \emph{MCD-finite domains and ascent of IDF-property in polynomial extensions}, Comm. Algebra \textbf{46} (2018) 3865--3872.

    \bibitem{lF70} L. Fuchs, \emph{Infinite Abelian Groups I}, Academic Press, 1970.
    
	\bibitem{GGT21} A. Geroldinger, F. Gotti, and S. Tringali, \emph{On strongly primary monoids, with a focus on Puiseux monoids}, J. Algebra \textbf{567} (2021) 310--345.
	
	\bibitem{GH06b} A.~Geroldinger and F.~Halter-Koch: \emph{Non-unique Factorizations: Algebraic, Combinatorial and Analytic Theory}, Pure and Applied Mathematics Vol. 278, Chapman \& Hall/CRC, Boca Raton, 2006.

	\bibitem{rG84} R.~Gilmer: \emph{Commutative Semigroup Rings}, Chicago Lectures in Mathematics, The University of Chicago Press, London, 1984.

    \bibitem{GLRRT24} V. Gonzalez, E. Li, H. Rabinovitz, P. Rodriguez, and M. Tirador, \emph{On the atomicity of power monoids of Puiseux monoids}. International Journal of Algebra and Computation (to appear). Preprint on arXiv: https://arxiv.org/abs/2401.12444
    
	\bibitem{fG19} F.~Gotti: \emph{Increasing positive monoids of ordered fields are FF-monoids}, J. Algebra \textbf{518} (2019) 40--56.

	\bibitem{GL23} F. Gotti and B. Li, \emph{Atomic semigroup rings and the ascending chain condition on principal ideals}, Proc. Amer. Math. Soc. \textbf{151} (2023) 2291--2302.
	
	\bibitem{GL23a} F. Gotti and B. Li, \emph{Divisibility and a weak ascending chain condition on principal ideals}. Submitted. Preprint on arXiv: https://arxiv.org/abs/2212.06213
	
	\bibitem{GL22} F. Gotti and B. Li, \emph{Divisibility in rings of integer-valued polynomials}, New York J. Math. \textbf{28} (2022) 117--139.
	
	\bibitem{GR23} F. Gotti and H. Rabinovitz, \emph{On the ascent of atomicity to monoid algebras}. Journal of Algebra (to appear). Preprint on arXiv: https://arxiv.org/abs/2310.18712.
	
	\bibitem{GV23} F. Gotti and J. Vulakh, \emph{On the atomic structure of torsion-free monoids}, Semigroup Forum \textbf{107} (2023) 402--423.
		
	\bibitem{aG74} A.~Grams, \emph{Atomic rings and the ascending chain condition for principal ideals}. Math. Proc. Cambridge Philos. Soc. \textbf{ 75} (1974) 321--329.

	\bibitem{mR93} M. Roitman, \emph{Polynomial extensions of atomic domains}, J. Pure Appl. Algebra \textbf{87} (1993) 187--199.
		
	\bibitem{aZ82} A. Zaks, \emph{Atomic rings without a.c.c. on principal ideals}, J. Algebra \textbf{74} (1982) 223--231.
\end{thebibliography}
\end{document}